\documentclass{amsart}

\usepackage[english]{babel}
\usepackage[utf8]{inputenc}
\usepackage{lmodern}
\usepackage{microtype}

\usepackage{mathtools}
\usepackage{amssymb,amsthm,amsfonts,amsmath}

\usepackage{enumitem}

\usepackage{tikz-cd}
\tikzcdset{arrow style=math font}

\numberwithin{equation}{section}

\newtheorem{theorem}{Theorem}[section]

\newtheorem{proposition}[theorem]{Proposition}
\newtheorem{corollary}[theorem]{Corollary}

\theoremstyle{definition}
\newtheorem{definition}[theorem]{Definition}

\theoremstyle{remark}
\newtheorem{remark}[theorem]{Remark}

\usepackage{hyperref}

\newcommand{\Z}{\mathbb{Z}}

\newcommand{\N}{\mathbb{N}}
\newcommand{\C}{\mathbb{C}}
\newcommand{\T}{\mathfrak{T}}
\newcommand{\A}{\mathfrak{A}}
\newcommand{\Ideal}{\mathfrak{I}}
\newcommand{\KK}{\mathfrak{KK}}
\newcommand{\K}{\mathrm{K}}
\newcommand{\B}{\mathfrak{B}}
\newcommand{\im}{\mathrm{i}}
\newcommand{\Abtwo}{\mathfrak{Ab}^{\Z/2}_{\textup c}}
\newcommand{\Modtwo}[1][\mathfrak{C}^{\textup{op}}]{\mathfrak{Mod}^{\Z/2}_{\textup c}#1}

\newcommand\restr[2]{\ensuremath{\left.#1\right|_{#2}}}

\newcommand*{\defeq}{\mathrel{\vcenter{\baselineskip0.5ex \lineskiplimit0pt
                     \hbox{.}\hbox{.}}}%
                     =}

\newcommand*{\nbd}{\nobreakdash}

\begin{document}

\title{Relative homological algebra for bivariant K-theory}
\author{George Nadareishvili}
\thanks{Supported by Shota Rustaveli National Science Foundation of Georgia, FR-18-10849.}
\thanks{George Nadareishvili, A.~Razmadze Mathematical Institute, gnadare@gwdg.de}

\begin{abstract}
This survey article on relative homological algebra in bivariant K\nbd-thoery is mainly intended for readers with a background knowledge in triangulated categories. We briefly recall the general theory of relative homological algebra in triangulated categories and latter specialize it to the non\nbd-equivariant and the equivariant bivariant K\nbd-thoery, where the actions on C*\nbd-algebras is by a finite cyclic group. We conclude by the explicit computation of the universal abelian invariant for separable C*\nbd-algebras with the action of~$\Z/4$ by automorphisms.
\end{abstract}

\maketitle

\section{Introduction} 

Given a locally compact Hausdorff space~$X$, one can consider the algebra of continuous functions from~$X$ to complex numbers vanishing at infinity. This construction is functorial and induces the contravariant equivalence between the category of locally compact Hausdorff spaces and the category of commutative C*\nbd-algebras with appropriate morphisms (see for example~\cite{WO93}). Subsequently, every property of a locally compact Hausdorff space can be
expressed in terms of the function algebra, formulation of which will then usually extend to any (noncommutative) C*\nbd-algebra. In this way, C*\nbd-algebra theory can
successfully be regarded as ``noncommutative topology''.

Similar to topological spaces, the key component of noncommutative topology is the study of C*\nbd-algebras using homology functors. However, in case of  C*\nbd-algebras, up to mild assumptions, there is a universal object among such theories. In particular, by the work of Higson~\cite{Hig87}, there exists a universal category $\mathfrak{KK}$, admitting a functor from the category of all separable C*\nbd-algebras, such that any stable, additive, split short exact sequence preserving functor into an additive category factors through $\mathfrak{KK}$. The similar statements (with respect to appropriate versions of $\mathfrak{KK}$) hold if we endow C*\nbd-algebras with an additional structure. Initially defined by Genadi Kasparov~\cite{Kas88} using functional analysis of Fredholm bimodules, the universal category\nbd-theoretic characterizations exhibit these categories (also referred to as Kasparov categories, bivariant K\nbd-theories or KK\nbd-theories) as an indispensable tool in the study of C*\nbd-algebras.

Arising as homotopy categories, Kasparov categories are naturally triangulated~\cite{MN06}. 

The techniques of triangulated categories proved fruitful for bivariant K\nbd-theory. Most notably in connection to classification programs~\cite{Del11, Del14, DM21, MN10, MN12, Nad19} and Baum\nbd--Connes conjecture~\cite{MN06}.

Classification in KK\nbd-theories usually involves applying the tools of relative homological algebra. It proceeds by functorially pushing computations from a triangulated domain to an abelian setting, where more tools are available to tackle a problem. Results of this nature include (but are not limited to) spectral sequences for derived functor computation, universal coefficient theorems and, consequentially, subcategory classifications.

In Section~\ref{sec:hom} we will recall the known facts from relative homological algebra in triangulated categories. The theory is exceptionally rich,  as demonstrated when first explored by of J. Daniel Christensen~\cite{Chr98} and Apostolos Beligiannis~\cite{Bel00}.  

In Section~\ref{sec:biv} we will overview the triangulated structure of equivariant bivariant K\nbd-theory and specialize the subject and techniques developed in Section~\ref{sec:hom}. The relative homological algebra for bivariant K-theory was developed by Ralf Meyer and Ryszard Nest in~\cite{Mey08, MN10}.

To demonstrate an example, we will conclude by an explicit computation of the universal abelian invariant for the KK-category of C*\nbd-algebras with the action of cyclic group of order $4$.

\section{Relative homological algebra in triangulated categories}\label{sec:hom}

The facts recalled in this section can be found in~\cite{MN10} and~\cite{Bel00}.

\begin{definition}
An additive category $\mathfrak{A}$ is called \emph{stable} if it is equipped with an automorphism \[\Sigma_{\mathfrak A}\colon \mathfrak{A}\to \mathfrak{A}.\]
\end{definition}
For example, given an abelian category $\mathfrak{A}$, we can construct a $\Z/k$\nbd-graded (we also allow $k=0$, by defining $\Z/0\defeq \Z$) stable abelian category $\mathfrak{A}^{\Z/k},$ by considering the category product $\mathfrak{A}^{\Z/k}=\prod_{i\in \Z/k} \mathfrak{A}$ and the suspension functor $\Sigma_{\mathfrak{A}^{\Z/k}}$ that shifts the $i$th component of objects (and morphisms) one place to the left. 

\begin{definition}
\emph{A stable homological functor} is a  homological functor into a stable abelian category that commutes with the suspension functor.
\end{definition}
For example, to a homological functor $H\colon\T\to \A$ from some triangulated category $\T$  into an abelian category $\A$, we can associate a stable homological functor $H_*\colon\T\to\A_{\Z/k}$ by defining 
\begin{equation} \label{eq:sta}
H_*(A)_i=H_i(A)\defeq H(\Sigma_{\A_{\Z/k}}^{-i}X)
\end{equation} 
for the $i$th component of the object $H_*(A)\in\A_{\Z/k}.$

Homological algebra in the non\nbd-abelian setting is always relative; that is, one needs an additional structure to get started. For triangulated categories, this structure can be given by fixing a class of morphisms.

\begin{definition}
  A collection of subgroups $\mathfrak{I}(A,B)\subseteq \mathfrak{T}(A,B)$ for all pairs of objects $A,B$  in a triangulated category $\mathfrak{T}$, such that \[\mathfrak{T}(C,D)\circ \mathfrak{I}(B,C)\circ \mathfrak{T}(A,B) \subseteq \mathfrak{I}(A,D),\] for all $A,B,C,D \in \mathfrak{T},$ is called an \emph{ideal} $\mathfrak{I}$ in $\T$.
  \end{definition} 
  \begin{remark}
  Equivalently, we can think of $\Ideal$ as fixing the subclass of exact triangles
  \[A\to B\to C\xrightarrow{f}\Sigma A\] with $f\in \Ideal(C,\Sigma A)$. These are called \emph{pure} triangles in~\cite{Bel00} ($\Ideal$\nbd-exact triangles below). For readers familiar with extriangulated categories, it should be remarked that under mild assumptions, pure triangles constitute the extriangulated subcategory, when one views $\T$ as an extriangulated category itself~\cite{HZD20}.
  \end{remark}
  As an example of an ideal, consider a homological functor $ H\colon \mathfrak{T} \rightarrow \mathfrak{A}$ into an abelian category~$\mathfrak{A}.$ Then the kernel of $H$ on morphisms
 \[ \ker H(A,B) = \{f \in \mathfrak{T}(A,B) \mid F(f)=0 \}\] clearly defines an ideal in $\T$.
 \begin{definition}
 An ideal $\mathfrak{I}$ is called \emph{homological} if it is the kernel of a stable homological functor.
\end{definition}
 Note that different functors can give rise to the same homological ideal by sharing a kernel. However, the resulting homological algebra will only depend on the ideal itself.
 
We will only deal with homological ideals. As elements of a kernel, the morphisms in $\Ideal$ can be thought of as vanishing relative to $\Ideal$ in $\T$. An application of the homological functor with $\ker H=\Ideal$ explains the following terminology.

\begin{definition}
We call the exact triangle \[ A \xrightarrow{\;f\;} B \xrightarrow{\;g\;}  C \xrightarrow{\;h\;} \Sigma A \]  \emph{$\Ideal$-exact} in $\T$, if $h \in \Ideal( C , \Sigma A)$. In this situation, $f$ is called \emph{$\Ideal$-monic} and $g$ is called \emph{$\Ideal$-epic}.
\end{definition}

Let $\Ideal=\ker H$ be a homological ideal.
\begin{definition}
We say that a chain complex $C_{\bullet}=(C_n,d_n)$ in $\T$ is \emph{$\Ideal$-exact} in degree~$n$ if
\[  H(C_{n+1}) \xrightarrow{H(d_{n+1})}  H(C_n) \xrightarrow{H(d_n)} H(C_{n-1}) \]
is exact at $H(C_n)$.
\end{definition}
We call a chain complex \emph{$\Ideal$-exact} if it is $\Ideal$-exact in every degree.

\subsection{Relative projective objects}

\begin{definition}
A homological functor $H\colon \mathfrak{T} \rightarrow \mathfrak{A}$ is called \emph{$\Ideal$-exact} if $\mathfrak{I} \subseteq \ker H$; that is, $H(f)=0$ for all $f\in \mathfrak{I}(A,B)$. 
\end{definition} 
 
As for abelian categories, we define
\begin{definition}
An object $A \in \mathfrak{T}$ is called \emph{$\mathfrak{I}$-projective} if the functor $\mathfrak{T}(A,-)$ is $\mathfrak{I}$-exact. 
\end{definition} 
 
Let $\mathfrak{P}_{\mathfrak{I}}$ denote the full subcategory of $\mathfrak{I}$\nbd-projective objects in $\T$. $\mathfrak{P}_\Ideal$ is closed under suspensions, desuspensions, retracts and whatever coproducts exist in $\T$.
 
 Having candidates for projective objects, we continue by constructing homological algebra in an analogy to the abelian case. 

\begin{definition}
Let $\Ideal$ be a homological ideal in $\T$ and $A \in \T$. We say that $\pi\colon P \rightarrow A$ is a \emph{one\nbd-step $\Ideal$\nbd-projective resolution} if $\pi$ is $\Ideal$\nbd-epic and $P \in \mathfrak{P}_\Ideal$. An \emph{$\Ideal$\nbd-projective resolution} of $A$ is an $\Ideal$\nbd-exact chain complex 
\[\cdots \rightarrow P_n \rightarrow P_{n-1} \rightarrow \cdots \rightarrow P_0 \rightarrow A\] with $P_n \in \mathfrak{P}_\Ideal$ for all $n \in \mathbb{N}$.
 \end {definition}
 
 We will say that there are \emph{enough $\Ideal$\nbd-projective objects} in $\T$ if every object $A \in \T$ has a one\nbd-step $\Ideal$\nbd-projective resolution. 
 
  Relative projective objects enjoy properties similar to projective objects in an abelian category.
 
 \begin{proposition}[Meyer-Nest {\cite[Proposition 3.26]{MN10}}]
 \label{pro:cha}
 Every object in $\T$ has an $\Ideal$\nbd-projective resolution if and only if $\T$ has enough $\Ideal$\nbd-projective objects. 
 
 Any map between objects of $\T$ can be lifted to a chain map between $\Ideal$\nobreakdash-projective resolutions of these objects and this lifting is unique up to chain homotopy. Two $\Ideal$\nbd-projective resolutions of the same object are chain homotopy equivalent.
 \end{proposition}
 
 This allows to define derived functors as in the classical case. Let $F \colon \T \to \A$ be an additive functor into an abelian category $\A$. Denote by $\mathfrak{Ho}(\T)$ and $\mathfrak{Ho}(\A)$ the categories of chain complexes up to chain homotopy of $\T$ and $\A$, respectively.  Applying $F$ termwise to a chain complexes induces a functor $\mathfrak{Ho}(F)\colon \mathfrak{Ho}(\T) \to \mathfrak{Ho}(\A).$ By Proposition~\ref{pro:cha}, construction of projective resolutions defines a
functor $P \colon \T \to \mathfrak{Ho}(\T)$. Denote by $\mathrm{H}_n \colon \mathfrak{Ho}(\A) \to \A$ the $n$th homology functor for $n \in \N$.
 
 \begin{definition}
   For an  additive functor $F \colon \T \to \A$, the composite 
   \[\mathbb{L}_nF\colon \T\xrightarrow{P}\mathfrak{Ho}(\T)\xrightarrow{\mathfrak{Ho}(F)}\mathfrak{Ho}(\A)\xrightarrow{\mathrm{H}_n}\A\]
is called the \emph{$n$th left derived functor of $F$. }

If $F\colon \T^{\mathrm{op}}\to \A$ is now a contravariant additive functor,
the composite 
\[\mathbb{R}^nF\colon \T^{\mathrm{op}}\xrightarrow{P}\mathfrak{Ho}(\T)\xrightarrow{\mathfrak{Ho}(F)}\mathfrak{Ho}(\A)\xrightarrow{\mathrm{H}^n}\A\]
is called the \emph{$n$th right derived functor of $F$.}
\end{definition}

Relative derived functors share many similar properties with their abelian counterparts. There is even a spectral sequence relating a homological functor and its relative derived functors. We are not going to discuss this general construction. We only recall the hereditary case of the Universal Coefficient Theorem, where the spectral sequence collapses to the short exact sequence.

Denote by $\langle \mathfrak{P}_\Ideal\rangle\subseteq\T$ subcategory generated by $\Ideal$\nbd-projective objects $\mathfrak{P}_\Ideal$; that is, the triangulated subcategory closed under whatever coproducts exist in $\T$ .
  
 \begin{theorem}[Meyer-Nest {\cite[Theorem 4.4]{MN10}}] \label{thm:guct}
 Let $\mathfrak{I}$ be a homological ideal in a triangulated category $\mathfrak{T}.$ Let $A\in\T$ have an $\Ideal$\nbd-projective resolution of length one. Suppose also that $A \in \langle \mathfrak{P}_\Ideal\rangle.$  Let $F \colon \T \to \A$ be a homological functor and $\tilde{F} \colon \T^{\textup{op}} \to \A$ a cohomological functor. Then there are natural short exact sequences
 \[0\to\mathbb{L}_0 F(A) \to F(A)\to \mathbb{L}_1 F(\Sigma A)\to 0\]
 \[0\to\mathbb{R}^1 \tilde{F}(\Sigma A) \to \tilde{F}(A)\to \mathbb{R}^0 \tilde{F}( A)\to 0\]
  \end{theorem}
  
  \begin{remark} \label{rem:uct}
  If, for example, we take $\tilde{F}=\T(\text{--},B)$ for any $B\in\T$ and denote $  \operatorname{Ext}_{\T,\Ideal}^n\defeq \mathbb{R}^n \tilde{F}$, under the assumptions of Theorem~\ref{thm:guct}, we get the short exact sequence
  \begin{equation}\label{eq:uct}
  0\to\operatorname{Ext}_{\T,\Ideal}^1(\Sigma A,B) \to \T(A,B)\to \operatorname{Ext}_{\T,\Ideal}^0(A,B)\to 0.
  \end{equation}
  \end{remark}
  
\subsection{The universal \texorpdfstring{$\Ideal$\nbd-}exact functor}

By the classical construction of Peter Freyd~\cite{Fre66}, any triangulated category admits a universal homological functor $U$ into an abelian category of finitely presented functors, such that any other homological functor uniquely factors through $U$ with an exact functor (unique up to natural isomorphism) between abelian categories.  

Notably, the  relative version of this statement is also true.

 \begin{definition} An $\mathfrak{I}$-exact homological functor $F$ is \emph{universal}, if any other $\mathfrak{I}$-exact homological functor $G\colon\mathfrak{T} \rightarrow \mathfrak{A'}$ factors as $G=\bar{G}\circ F$ with an exact functor $\bar{G}\colon\mathfrak{A} \rightarrow \mathfrak{A'}$ that is unique up to natural isomorphism.
 \end{definition}
  
 \begin{theorem}[Beligiannis {\cite[Section 3]{Bel00}}]
  For every homological ideal $\Ideal$ in a triangulated category $\T$, there exists an abelian category $\mathcal{A}_{\Ideal}(\T)$ and a universal $\mathfrak{I}$-exact stable homological functor $F\colon \T \rightarrow \mathcal{A}_{\Ideal}(\T)$.
 \end{theorem}

 Here, if no set\nbd-theoretic issues arise, the category $\mathcal{A}_{\Ideal}(\T)$ can be obtained by localizing the category of finitely presented functors at the Serre subcategory, where we quotient out all morphisms coming from the ideal $\Ideal$. 
 
 Under nonrestrictive assumptions, the universal $\mathfrak{I}$\nbd-exact homological functor identifies homological algebra in a target abelian category to the $\mathfrak{I}$\nbd-relative homological algebra in the domain triangulated category and allows the computation of relative derived functors using derived functors in the universal abelian category. More precisely, the following results hold.
 
 \begin{theorem}[Beligiannis {\cite[Proposition 4.19]{Bel00}}] \label{thm:bel}
 Let $\mathfrak{I}$ be a homological ideal in a triangulated category $\mathfrak{T}$ and let $F\colon \mathfrak{T} \rightarrow \mathfrak{A}$ be a universal $\mathfrak{I}$\nbd-exact  homological functor into an abelian category $\mathfrak{A}$. Suppose that idempotent morphisms in $\T$ split and that there are enough $\Ideal$\nbd-projective objects in $\T$. Then there are enough projective objects in $\A$ and $F$ induces an equivalence between the full subcategories of $\mathfrak{I}$\nbd-projective objects in $\mathfrak{T}$ and of projective objects in $\mathfrak{A}$.
 \end{theorem}
 
 \begin{theorem}[Meyer-Nest {\cite[Theorem 3.41]{MN10}}] \label{thm:uni}
 Let $\mathfrak{I}$ be a homological ideal in a triangulated category $\mathfrak{T}$ and let $F\colon \mathfrak{T} \rightarrow \mathfrak{A}$ be a universal $\mathfrak{I}$\nbd-exact (stable) homological functor into an abelian category $\mathfrak{A}$. Suppose that idempotent morphisms in $\T$ split and that there are enough $\Ideal$\nbd-projective objects in $\T$.
If $G\colon \T \to \A'$ is any (stable) homological functor, then there is a unique right exact (stable) functor $\bar{G} \colon \A \to \A'$ such that $\bar{G}\circ F(P) = G(P)$ for all $P \in \mathfrak{P}_{\Ideal}$.

The left derived functors of $G$ with respect to $\Ideal$ and of $\bar{G}$ are related by natural isomorphisms $\mathbb{L}_n\bar{G} \circ F(A) = \mathbb{L}_nG(A)$ for all $A \in T, n \in \N$. There is a similar statement for cohomological functors, which specializes to natural isomorphisms
\[ \operatorname{Ext}_{\T,\Ideal}^n(A,B)\cong \operatorname{Ext}_{\A}^n(F(A),F(B)).\]
 \end{theorem}
 
 \begin{remark} \label{rem:uct2}
 In light of Theorem~\ref{thm:uni}, once we have a universal $\Ideal$\nbd-exact homological functor $F\colon\T\to\A$, under the assumptions of Theorem~\ref{thm:guct}, the exact sequence~\eqref{eq:uct} takes the form
\[  0\to\operatorname{Ext}_{\A}^1(F(\Sigma A),F(B)) \to \T(A,B)\to \operatorname{Hom}_{\A}(F(A),F(B))\to 0.\]
\end{remark}

\subsection{An example of the universal \texorpdfstring{$\Ideal$\nbd-}exact functor}\label{sub:uni}

Now, we construct for us the most relevant example of the universal $\Ideal$\nbd-exact stable homological functor.

Fix at most countable set of objects $\mathcal{C}$ in a triangulated category $\mathfrak{T}$ with countable coproducts. Denote by $\mathfrak{I}_{\mathcal{C}}$ the homological ideal defined as the kernel of the functor \[F_{\mathcal{C}}\colon\mathfrak{T} \rightarrow \prod_{C \in \mathcal{C}}\mathfrak{Ab}^{\mathbb{Z}}, \qquad A \mapsto \big(\mathfrak{T}(C,A)\big)_{C \in \mathcal{C}}.\] 
Assume that $F_{\mathcal{C}}(A)$ is countable for all $A \in \mathfrak{T}$.
 
 Let $\mathfrak{C}$ denote a $\mathbb{Z}\,$\nbd-graded pre\nbd-additive full subcategory of $\mathfrak{T}$ on objects $\mathcal{C}$. 
 Denote by $\mathfrak{Mod}(\mathfrak{C}^{\mathrm{op}})_{\textup c}$ the category of functors with countable values $\operatorname{Funct}(\mathfrak{C}^\mathrm{op}, \mathfrak{Ab}^{\Z})_{\textup{c}}$, or equivalently the category of countable graded right modules over the category ring $R^{\mathfrak C}$ of $\mathfrak C$. Then the enrichment of $F_{\mathcal{C}}$ to the functor \[F_{\mathfrak{C}}\colon \mathfrak{T} \rightarrow \mathfrak{Mod}(\mathfrak{C}^{\mathrm{op}})_\textup{c},\] 
 with the right $\mathfrak{C}$-module structure on $\big(\mathfrak{T}(C,A)\big)_{C \in \mathfrak{C}}$ coming from composition of morphisms in $\mathfrak{T}$, is the universal $\mathfrak{I}_{\mathfrak{C}}$-exact stable homological functor~\cite[Theorem 4.4]{MN12}.

\section{Universal invariants for bivariant K-theory}\label{sec:biv}
 
 In what follows, we will show two examples of applications of relative homological algebra to noncommutative topology. 
 
 Both examples are of KK-theory viewed as a category, one with extra structure. KK\nbd-theory is a joint generalization of topological K\nbd-theory and K\nbd-homology for noncommutative spaces. However, keeping the audience in mind, we will define the category in question rather unconventionally, using the universal property mentioned in the introduction and not the more traditional, Fredholm bimodule picture. 
 
 \begin{definition}
 A \emph{C*\nbd-algebra} $A$ is a Banach algebra over complex numbers together with a conjugate linear automorphism ${}^*\colon A\to A$ called involution, such that $(ab)^*=b^*a^*$, $(a^*)^*=a$ and $\lVert a^*a\rVert=\lVert a^*\rVert\lVert a\rVert$ for all  $a,b\in A$.
 
 A ring homomorphism of C*\nbd-algebras that also preserves involution is called a *-homomorphism. 
 \end{definition}
 
 We denote by $\mathfrak{C^*alg}$ the category of C*\nbd-algebras as objects and *-homomorphisms as arrows.
 
 Let $G$ be a locally compact group.
 
 \begin{definition}
  A \emph{$G$\nbd-C*\nbd-algebra} $A$ is  a strongly continuous representation of $G$ by *\nbd-automorphisms $G\to \mathrm{Aut}(A)$.
\end{definition} 
 
 By $G$\nbd-$\mathfrak{C^*alg}$ we denote the category with $G$\nbd-C*\nbd-algebras as objects and $G$~equivariant *\nbd-homomorphisms as arrows. Whenever $G$ is a trivial group we recover the category $\mathfrak{C^*alg}$, so it is sufficient to state definitions for only the equivariant case. For basic properties of the categories of C*\nbd-algebras we refer to~\cite{Dav96} and~\cite{Mey08a}.
 
 As for rings, an \emph{extension} of $G$\nbd-C*\nbd-algebras is a diagram isomorphic to $I\to A\to A/I$ in $G$\nbd-$\mathfrak{C^*alg}$ for some $G$\nbd-invariant ideal $I$ in a $G$\nbd-C*\nbd-algebra $A$. We call an extension \emph{split}, if it has a section in $G$\nbd-$\mathfrak{C^*alg}$. 
 
 Let $\mathfrak{A}$ denote an exact category. If $\mathfrak{A}$ is only additive to begin with, we can endow it with trivial exact category structure with all extensions being split.
 
 \begin{definition}
 A functor $F\colon G\text{-}\mathfrak{C^*alg}\to \mathfrak{A}$ is called \emph{split exact}, if for any split extension $A\xrightarrow{i}B\xrightarrow{p}C $ with section $s\colon C \to B$, the map $\big(F(i),F(s)\big)\colon F(A)\oplus F(C)\to F(B)$ is an isomorphism.
 \end{definition}
 
 We will also use the definition of C*\nbd-stability. For this, we need a monoidal structure on $G$\nbd-$\mathfrak{C^*alg}$. Nevertheless, we will not define it here. We only note that there are two reasonable ways to complete the algebraic tensor product of $G$\nbd-C*\nbd-algebras $A$ and $B$, called \emph{minimal} and \emph{maximal} tensor products respectively. However, for the large class of C*\nbd-algebras these two constructions coincide; we will use the notation $A\otimes B$ whenever this is the case (for details see \cite{Mur90} and \cite{Was94}).
 
 More intuitive definition of C*\nbd-stability is in a non\nbd-equivariant setting. 
 \begin{definition}
 For a rank\nbd-one projection $p\in \mathbb{K}({\ell^2\N})$ in compact operators on $\ell^2\N$, an embedding $i\colon A\to A\otimes \mathbb{K}({\ell^2\N})$ given by $i(a)=a\otimes p$, is called a \emph{corner embedding} of $A$.
 
 A functor $F\colon \mathfrak{C^*alg}\to \mathfrak{A}$ is called \emph{C*\nbd-stable} if any corner embedding induces an isomorphism $F(A)\cong F\big(A\otimes \mathbb{K}({\ell^2\N})\big)$.
 \end{definition} 
 
 The appropriate generalization of  C*\nbd-stability to the equivariant case is the following: 
 
 \begin{definition}
 Given canonical embeddings of any non-zero $G$\nbd-Hilbert spaces $\mathcal{H}_1\rightarrow\mathcal{H}_1\oplus \mathcal{H}_2\leftarrow \mathcal{H}_2,$ a functor $F\colon G\text{\nbd-}\mathfrak{C^*alg}\to \mathfrak{A}$ is \emph{C*\nbd-stable} if it induces the isomorphisms
 \[ F\big(A\otimes\mathbb{K}(\mathcal{H}_1)\big)\xrightarrow{\cong} F\big(A\otimes\mathbb{K}(\mathcal{H}_1\oplus \mathcal{H}_2)\big)\xleftarrow{\cong} F\big(A\otimes\mathbb{K}(\mathcal{H}_2)\big).\]
 \end{definition}
 
 Now we are ready to give a universal category\nbd-theoretic definition of KK-theory. For technical reasons, we will restrict attention to a full subcategory $G$\nbd-$\mathfrak{C^*sep}\subseteq G\text{\nbd-}\mathfrak{C^*alg}$ of \emph{separable} $G$\nbd-C*\nbd-algebras: equivariant C*\nbd-algebras with a countable dense subsets.
 
 \begin{theorem}[Higson {\cite[Theorem 4.5]{Hig87}}]
 There exists the additive category $\mathfrak{KK}^G$ and the universal split-exact C*\nbd-stable functor $\mathrm{KK}^G\colon G\text{\nbd-}\mathfrak{C^*sep}\to \mathfrak{KK}^G.$ 
 \end{theorem}
 
 In other words, any C*\nbd-stable and split-exact functor $F\colon G\text{\nbd-}\mathfrak{C^*alg}\to \mathfrak{A}$ will factor uniquely through $\mathfrak{KK}^G$.
 
 \begin{remark}
 Of course, the existence and the universal property defines $\mathfrak{KK}^G$ up to the category theoretic equivalence. We will not need the original, admittedly more practical definition of $\mathfrak{KK}^G.$ We only state that the objects of $\mathfrak{KK}^G$ are the same as of $G\text{\nbd-}\mathfrak{C^*sep}$, namely the separable $G$\nbd-C*\nbd-algebras, and the morphisms are described concretely as the $G$\nbd-equivariant Kasparov group $\mathfrak{KK}^G(A,B)=\mathrm{KK}_0^G(A,B)$ for $A,B\in \mathfrak{C^*sep}.$ The composition is given by the so called \emph{Kasparov product.} We direct the reader interested to learn more about KK-theory to the textbook sources like~\cite{Bla98} or the original paper by Genadi Kasparov~\cite{Kas88}.
 \end{remark}
 
 \subsection{Triangulated structure} 
 As already stated, the category $\mathfrak{KK}^G$ is additive. Coproduct is given by direct sum of $G$\nbd-C*\nbd-algebras. 
 
 Consider the functor
 \[\Sigma \colon \mathfrak{KK}^G \rightarrow \mathfrak{KK}^G \qquad A \mapsto  C_0(\mathbb{R}) \otimes A.\]
 As a consequence of a remarkable theorem by Raul Bott, $\mathrm{KK}^G$ (and thus any C*\nbd-stable and split exact functor) satisfies Bott periodicity, that is, in $\mathfrak{KK}^G$ there are natural isomorphisms
$\Sigma^2(A)\cong A$ for all $A \in \mathfrak{KK}^G$. Therefore $\Sigma$ is an an automorphism up to a natural isomorphism and thus, $\mathfrak{KK}^G$ is stable. We refer to $\Sigma$ as the suspension.  

Let now $A \rightarrow B \rightarrow C$ be an extension of $G$\nbd-C*\nbd-algebras. This extension is called \emph{cp\nbd-split} if there is a $G$\nbd-equivariant, completely positive (see~\cite{Dav96}), contractive section $C \rightarrow B$.

In analogy to topological spaces, the \emph{cone} of a morphism $A \xrightarrow{f} B$ between $G$\nbd-C*\nbd-algebras is defined as 
\[\operatorname{cone}(f)\defeq \{(a,b) \in A \times C_0((0,1],B) \mid f(a)=b(1) \}.\] 

For every cp\nbd-split extension $A \rightarrow B \rightarrow C$, with $A,B,C$ separable $G$\nbd-C*\nbd-algebras, there is a unique $G$\nbd-equivariant map $\Sigma C \rightarrow A$ and an isomorphism $A \xrightarrow{\cong} \operatorname{cone}({B \rightarrow C})$ in $\mathfrak{KK}^G$, such that the following diagram commutes:
           
\[\begin{tikzcd}
\Sigma C \arrow[r]\arrow[d,dash,shift left=.1em] \arrow[d,dash,shift right=.1em]& A\arrow[r]\arrow[d,"\cong"] &B\arrow[d,dash,shift left=.1em] \arrow[d,dash,shift right=.1em]\arrow[r] &C \arrow[d,dash,shift left=.1em] \arrow[d,dash,shift right=.1em] \\
\Sigma C \arrow[r]& \operatorname{cone}(B\to C)\arrow[r] &B\arrow[r] &C 
\end{tikzcd}
               \]
 The first row in the above diagram is called the $G$\nbd-equivariant \emph{extension triangle} of the cp\nbd-split extension $A \rightarrow B \rightarrow C$. 
 
 Now, declare all $4$\nbd-term diagrams in $\mathfrak{KK}^G$ isomorphic to the extension triangle of some cp\nbd-split extension as exact triangles in $\mathfrak{KK}^G$. This way $\mathfrak{KK}^G$ becomes a triangulated category. However, we are faced with the notation problem. As constructed, mapping cone triangles have the form \[\Sigma C\to \operatorname{cone}(f)\to B\xrightarrow{f} C,\] so the arrows point in the opposite direction to the established conventions of triangulated categories. This is explained by the fact that the functor $X\to C_0(X)$ from locally compact spaces to C*\nbd-algebras is contravariant. Nevertheless, this is not really a problem, as in general, $\T^{\mathrm{op}}$ inherits triangulated structure from $\T$ canonically, and moreover $\Sigma\cong \Sigma^{-1}$ in our case.
 
  More details on triangulated structure in $\mathfrak{KK}^G$ can be found in~\cite{MN06}.
 \begin{remark}
 $\mathfrak{KK}^G$ only contains countable coproducts. So, in following subsections, when we will talk about localizing subcategories, we will really mean localizing$_{\aleph_1}$ subcategories as in~\cite{Del10}, that is, triangulated subcategories closed under countable coproducts.
 \end{remark}

\subsection{The bootstrap class example}

We start by the observation that morphism sets in $\mathfrak{KK}^G$ are closely related to K\nbd-theory.
For a compact group $G$, denote by $\mathrm{K}^G_0$ the $G$\nbd-equivariant topological K\nbd-theory. Generalizing Atiyah-Segal's $G$\nbd-equivariant vector bundle K\nbd-cohomology of topological spaces, for any $G$\nbd-C*\nbd-algebra $A$, there is the natural isomorphism~\cite{Bla98} \begin{equation} \label{eq:ati}
\mathfrak{KK}^G(\C,A)=\mathrm{KK}^G_0(\C,A)\cong \mathrm{K}^G_0(A).
\end{equation}
Thus, the natural challenge of noncommutative topology is to compute $\mathrm{KK}$ groups using K\nbd-theoretic invariants. Probably the most famous result in this context is the Universal Coefficient Theorem by Rosenberg and Schochet~\cite{RS87}. We will try to describe this theorem by relative homological algebra. 

Following~\eqref{eq:sta}, we define 
\begin{definition} \label{def:gra} 
 \[\mathrm{KK}^G_n(A,B)\defeq \mathrm{KK}^G_0(A,\Sigma^n(B))\]
 \end{definition}
 Therefore, also $\mathrm{K}^G_n(B)\defeq\mathrm{K}^G_0(\Sigma^n(B))$ by~\eqref{eq:ati}. Since $\Sigma^2\cong \operatorname{Id},$ we get a $\Z/2$\nbd-graded abelian theory in both cases.

Going forward in this subsection, we assume that $G$ is trivial. So, the situation is non\nbd-equivariant and we simply write $\KK$ for the universal split\nbd-exact, C*\nbd-stable category for separable C*\nbd-algebras.

To proceed, we need to restrict to a smaller subcategory of  C*\nbd-algebras in $\KK$.

\begin{definition}
 The \emph{bootstrap class} $\mathfrak{B} \subset \KK$ is the localizing triangulated subcategory in $\KK$ generated by the object $\mathbb{C}\in \KK$; that is, $\mathfrak{B} = \langle \mathbb{C} \rangle.$
\end{definition}
 
Another equivalent way to characterize the bootstrap class is all separable C*\nbd-algebras that are isomorphic to commutative C*\nbd-algebras in $\KK$. The class is large, and most separable C*\nbd-algebras that operator algebraist encounters in daily work are in fact in $\mathfrak{B}$.
 
 Now, fix the generator, the complex numbers, as the single object $\C\in\B$ and consider a $\Z/2$\nbd-graded  representable functor 
 \[\mathrm{KK}_*(\C,B)=\big( \mathrm{KK}_0(\C,B), \mathrm{KK}_1(\C,B)\big)\cong \big( \mathrm{K}_0(B), \mathrm{K}_1(B)\big)=\mathrm{K}_*(B)\] into $\Abtwo,$ the category of $\Z/2$\nbd-graded countable abelian groups with degree preserving homomorphisms. Here countability condition comes from the fact that we are only considering separable C*\nbd-algebras. 
 
 In the notation of Subsection~\ref{sub:uni}, we took $\mathfrak{C}=\{\C\}$, and thus $\Modtwo \cong\Abtwo$ by definition. Therefore, we are left with the K\nbd-theory functor
 \[\mathrm{K}_*\colon \B \longrightarrow \Abtwo \qquad A\mapsto \mathrm{KK}_*(\C,\text{-})\]
 which is the universal stable $\ker \mathrm{K}_*$\nbd-exact functor. Therefore, by Theorem~\ref{thm:uni} the relative derived functors can be computed using honest derived functors in hereditary abelian category $\Abtwo.$  So, Theorem~\ref{thm:guct} and Remark~\ref{rem:uct2}, we derive the celebrated Universal Coefficient Theorem:
 \begin{theorem}[Rosenberg-Schochet~\cite{RS87}] 
 Let $A$ be a separable $\mathrm{C^*}$-algebra. Then for $A \in \mathfrak{B}$,  there is a short exact sequence of $\Z/2$-graded abelian groups
 \[\mathrm{Ext}^1\big(\K_{*+1}(A),\K_*(B)\big)\rightarrowtail \mathfrak{KK}_*(A,B) \twoheadrightarrow \mathrm{Hom}\big(\K_*(A),\K_*(B)\big)\] for every $B \in \KK$. 
  \end{theorem}
\begin{remark}
The converse of the theorem is also true, that is, $A \in \mathfrak{B}$ only if the short exact sequence exists for every $B \in \KK$. This is sometimes also used to define the bootstrap class.
\end{remark}
 
 The Universal Coefficient Theorem is very useful as it allows the computation of $\mathrm{KK}$ groups using $\K$\nbd-theory for C*\nbd-algebras in the bootstrap class. This is widely used for classification programs of C*\nbd-algebras or even of different triangulated subcategories as in~\cite{Del11,DM21,Nad19}.
 
 \subsection{Actions of finite groups}
 
 Now we will give an example of the case when the category ring of $\mathfrak{C}$ is not hereditary and some computation is needed to pin down the universal invariant.
 
 Throughout this subsection, let $G$ be a finite group. As before, looking at whole $\KK^G$ is far too complicated, so we will restrict our attention to a smaller subclass of $G$\nbd-C*\nbd-algebras. The correct equivariant bootstrap class is defined as follows~\cite{DEM14}.
 \begin{definition}
 A $G$\nbd-C*\nbd-algebra $A$ is called \emph{elementary} if it is of the form 
 \[\operatorname{Ind}_H^G \mathbb{M}_n\C = \{G\xrightarrow{f}\mathbb{M}_n\C \mid hf(xh)=f(x)\text{ for any } x\in G \text{ and } h\in H\}\] with the $G$\nbd-action $(g f)(x)\defeq f(g^{-1}x)$, $g, x \in G,$ for some subgroup $H \subseteq G$ and some action by automorphisms of $H$ on $n\times n$ matrix algebra $\mathbb{M}_n\C$.
 \end{definition}
 
 \begin{definition}
 The \emph{$G$\nbd-equivariant bootstrap class} $\mathfrak{B}^G \subset \KK^G$ is the localizing triangulated subcategory generated by the elementary $G$\nbd-C*\nbd-algebras in $\KK^G$, that is, 
 \[\mathfrak{B}^G = \langle \operatorname{Ind}_H^G \mathbb{M}_n\C \mid \text{ For all }H\subseteq G\text{ actions on } \mathbb{M}_n\C  \rangle.\]
 \end{definition}
 
 Contrary to $\B$, equivariant bootstrap class is strictly larger than all $G$\nbd-C*\nbd-algebras isomorphic to commutative $G$\nbd-C*\nbd-algebras in $\KK^G$. The latter is too restrictive in the equivariant setting, as it is not even thick as a subcategory. On the other hand, $\mathfrak{B}^G$ as defined above is fairly large, as it is equivalent to all $G$\nbd-C*\nbd-algebras that are isomorphic in $\KK^G$ to a $G$\nbd-action on a type I C*\nbd-algebra (for the definition of Type I see for example~\cite{Bla98} or~\cite{Dav96}).
 
 By Skolem-Noether theorem, each automorphism of complex matrix algebra is inner, thus $\operatorname{Aut}(\mathbb{M}_n\C)\cong \operatorname{GL}_n(\C)/\C\mathrm{I_n}$. Therefore, every action of $H\subseteq G$ on $\mathbb{M}_n\C$ by automorphisms comes from some $n$\nbd-dimensional projective representation
 \[H\to \operatorname{GL}_n(\C)/\C\mathrm{I_n},\]
 which in turn are classified by a cohomology class in $\mathrm{H}^2(H,\mathrm{U}(1))$~\cite{Kar85}. Two actions on $\mathbb{M}_n\C$ are isomorphic in $\KK^G$ if and only if they belong to the same class in $\mathrm{H}^2(H,\mathrm{U}(1)).$ For a finite group, second cohomology is also finite~\cite{Kar85}, thus there is a finite choice (each per cohomology class of a projective representation) of elementary C*\nbd-algebras that generate $\B^G$.
 
 \subsubsection{Actions of finite cyclic groups}
 Now we further assume that $G=C_k,$ the cyclic group of order $k$. Then it is well known that there are no non\nbd-trivial projective representations for $H\subseteq C_k,$~\cite{Kar85} and thus $\operatorname{Ind}_H^{C_k}\C = C(C_k/H) $ are enough to  generate $\B^{C_k}$. 
 
 \begin{remark}\label{rem:res}
 In general, for a subgroup $H\subseteq G$, the construction that assigns $G$\nbd-C*\nbd-algebra $\operatorname{Ind}_H^{G}A$ to the $H$\nbd-C*\nbd-algebra $A$  is functorial. It is the left adjoint to the functor $\operatorname{Res}_H^{G}\colon \KK^G\to\KK^H$,  that restricts a $G$\nbd-action to $H$. So, when computing Yoneda functors represented by generators of $\B^{C_k}$, by~\eqref{eq:ati} we get an equivariant K\nbd-theory
 \[ \mathrm{KK}^G(C(G/H), A)= \mathrm{KK}^G(\operatorname{Ind}_H^{G}\C,A)\cong \mathrm{KK}^H(\C, \operatorname{Res}_H^{G} A)\cong \mathrm{K}_0^H(\operatorname{Res}_H^{G} A).\]
 \end{remark}
 
 Following Subsection~\ref{sub:uni}, to apply the machinery of relative homological algebra, we want to compute the category ring $R^{C_k}$ of the full subcategory $\mathfrak{C}^{C_k}\subset \KK^{C_k}$ on objects \[\{ C(C_k/H)\mid H\subseteq C_k\}.\] 
 
 As knowledgeable reader might have noticed, this looks like the domain for Mackey theory (for Mackey functors see~\cite{Bou97} or a shorter guide~\cite{Web00}). This is indeed the case and the following computation falls completely under Ivo Dell'Ambrogio's work~\cite{Del14}. Unpacking~\cite[Theorem~4.9]{Del14} and adapting it to ring theoretic conventions, we find that the category ring $R^{C_k}$ of $\mathfrak{C}^{C_k}$ is generated by the arrows of the form
 \begin{align*}
    r^H_L &\colon C(C_k/H) \rightarrow C(C_k/L), \\
    i^{H}_{L} &\colon C(C_k/L) \rightarrow C(C_k/H), \\
    c_g^{H} &\colon C(C_k/H) \rightarrow C(C_k/H), \\
    m_{\chi}^{H} &\colon C(C_k/H) \rightarrow C(C_k/H),
\end{align*}
 for all $L\subseteq H\subseteq C_k, g\in C_k$ and complex group  representation $\chi\in\mathrm{Rep}(H)$ of $H$. These are called \emph{ restriction, induction, conjugation, multiplication} respectively and are subject to the relations that follow.
 
 The first six are well known relations from Mackey theory. Let $H\subseteq  C_k$. 
 \begin{enumerate}[label=\textup{(\roman*)}] 
    \item \label{rel:one}\(r^{H}_H=i^{H}_H=1_{C(C_k/H)},\)
    
    \item \label{rel:two}\(c_h^{H}=1_{C(C_k/H)},\) if \(h \in H\),
    
    \item \label{rel:three} \(r^{K}_L \circ r^{H}_{K} =r^{H}_L,\) and \( i^{H}_K \circ i^{K}_L =i^{H}_L\) for \(L\subseteq K\subseteq  H,\)  
    
    \item \label{rel:four}\( c_g^{H}\circ c_h^{H}=c_{hg}^{H},\) for any $g,h\in C_k,$
    
    \item \label{rel:five} \(c_g^{K}\circ r^{H}_{K}  =  r^{H}_{K} \circ c_g^{H}\) and \(i^{H}_{K} \circ c_g^{K} =  c_g^{H} \circ i^{H}_{K}\) for $K\subseteq  H,$
    
    \item \label{rel:six} \( r^{H}_K \circ i^{H}_L=\sum_{g\in[L\backslash H /K]}  i^{K}_{L\cap K}\circ  c_g^{L\cap K}  \circ r_{L\cap K}^{L} \) for \(L,K\subseteq H.\)
    \end{enumerate}

    Then come the two relations pertaining to the multiplication, addition, restriction and conjugation of representations. Let~$\chi$ be a complex group representation of $H$. 
    \begin{enumerate}[resume, label=\textup{(\roman*)}]
    
    \item \label{rel:seven} \( m^{H}_{\chi}\circ m^{H}_{\psi} = m^{H}_{\chi\psi},\) \( m^{H}_{\chi}+ m^{H}_{\psi} = m^{H}_{\chi+\psi},\) and \(m^{H}_{\tau}=1_{C(C_k/H)}\) for every complex representation $\psi$ of $H$ and a trivial representation $\tau$ of $H$. 
   
\item \label{rel:eight} \( m^{L}_{\restr{\chi}{L}}\circ r^{H}_{L} = r^{H}_L\circ m_{\chi}^{H}\) for $L\subseteq H,$ 

\item \label{rel:nine}\(c^{H}_g \circ m^{H}_{\chi}= m_{\chi}^{H}\circ c^{H}_g.\) 

\end{enumerate}
    And finally, we have the so called \emph{Frobenius isomorphisms.} These reflect the ring structure on Mackey functors.
    \begin{enumerate}[resume, label=\textup{(\roman*)}]
    
    \item \label{rel:ten} \( m^{H}_{\chi} \circ i^{H}_{L} =  i^{H}_{L} \circ m^{L}_{\restr{\chi}{L}}\) for $L \subseteq  H,$

\item \label{rel:eleven} \( i^{H}_{L} \circ m^{L}_{\chi}\circ r^{H}_{L} = m_{\operatorname{ind}_L^H \chi}^{H}\) for $L\subseteq  H$, where $\operatorname{ind}_L^H \chi$ denotes the induced representation of $\chi$ to $H$.
    \end{enumerate}
   
   \subsubsection{Actions of the group $C_4$} \label{sub:act}
    The first cyclic group with non\nbd-trivial subgroup structure is $C_4=\{e,a,a^2,a^3\}$. 
    
    There is the single non\nbd-trivial subgroup $\langle a^2\rangle\subseteq C_4$ and thus three generators $C(C_4/\{e\})=C(C_4),\, C(C_4/\langle a^2\rangle)$ and $C(C_4/C_4)=\C$ for the equivariant bootstrap class $\B^{C_4}$. 
    
    $C_4$ has three non\nbd-trivial $1$\nbd-dimensional complex representations $\chi_{\im}^{C_4},$ $\chi_{-1}^{C_4}=(\chi_{\im}^{C_4})^2$ and $\chi_{-\im}^{C_4}=(\chi_{\im}^{C_4})^3$. Here in the subscript we indicate the image of the generator $a\in C_4$ in~$\C$. $\chi_{\im}^{C_4}$ and $\chi_{-\im}^{C_4}$ restrict to the only non\nbd-trivial representation of the subgroup $\langle a^2\rangle,$ which we denote by $\chi^{\langle a^2\rangle}.$
    
 By relations~\ref{rel:one} through~\ref{rel:four} and relation~\ref{rel:seven}, the eleven generators of the category ring $R^{C_4}$ are 
 \[
 \Gamma=\big\{
 1_{C(C_4)},\ 
 c_a^{\{e\}},\ 
 i^{\langle a^2\rangle}_{\{e\}},\ 
 r^{\langle a^2\rangle}_{\{e\}},\ 
 1_{C(C_4/\langle a^2\rangle)},\ 
 c_a^{\langle a^2\rangle},\ 
 m^{\langle a^2\rangle},\ 
 i_{\langle a^2\rangle}^{C_4},\ 
 r_{\langle a^2\rangle}^{C_4},\ 
 1_{\C},\ 
 m_{\im}^{C_4} 
 \big\}.
  \]
 
 The following represents how these generators act.
 \[
    \begin{tikzcd}[column sep = 8em]
    C(C_4)\arrow[r, "i^{\langle a^2\rangle}_{\{e\}}", shift left] \arrow[loop above, "1_{C(C_4)}{,\  }c_a^{\{e\}}"] & C(C_4/\langle a^2\rangle) \arrow[r, "i_{\langle a^2\rangle}^{C_4}", shift left] \arrow[l,"r^{\langle a^2\rangle}_{\{e\}}", shift left] \arrow[loop above, "1_{C(C_4/\langle a^2\rangle)}{,\ }c_a^{\langle a^2\rangle}{,\ }m^{\langle a^2\rangle}"]& \C \arrow[l,"r_{\langle a^2\rangle}^{C_4}", shift left] \arrow[loop above, "1_{\C}{,\ }m_{\im}^{C_4}"]
    \end{tikzcd}
\]

Since $R^{C_4}$ is a category ring, by definition we have that
\begin{enumerate}[start=0]
\item \label{c4rel:zero} for all $x,y\in\Gamma,$ if $x$ and $y$ are not composable in $\mathfrak{C}$, then $xy=0$.
\end{enumerate}
Relations~\ref{rel:two},\ref{rel:four} and~\ref{rel:seven} say that 
\begin{enumerate}[resume]
\item \label{c4rel:one} $(c_a^{\{e\}})^4=1_{C(C_4)}$,\quad $  
 (c_a^{\langle a^2\rangle})^2=(m^{\langle a^2\rangle})^2=1_{C(C_4/\langle a^2\rangle)}$,\quad
 $ (m_{\im}^{C_4})^4=1_{\C}.$
\end{enumerate}
Commutation relations~\ref{rel:five}, \ref{rel:eight}, \ref{rel:nine} and \ref{rel:ten}, together with \ref{rel:one}, \ref{rel:two} and \ref{rel:seven} imply that
\begin{enumerate}[resume]
\item \label{c4rel:two} 
$r^{\langle a^2\rangle}_{\{e\}} c_a^{\langle a^2\rangle} =c_a^{\{e\}}r^{\langle a^2\rangle}_{\{e\}}$,\quad $ c_a^{\langle a^2\rangle}r_{\langle a^2\rangle}^{C_4}=r_{\langle a^2\rangle}^{C_4}$,\quad $i^{\langle a^2\rangle}_{\{e\}}c_a^{\{e\}}= c_a^{\langle a^2\rangle}i^{\langle a^2\rangle}_{\{e\}}$,\\ 
$ i_{\langle a^2\rangle}^{C_4} c_a^{\langle a^2\rangle}=i_{\langle a^2\rangle}^{C_4};$
\quad $m^{\langle a^2\rangle} i^{\langle a^2\rangle}_{\{e\}}=i^{\langle a^2\rangle}_{\{e\}}$,\quad $r^{\langle a^2\rangle}_{\{e\}}m^{\langle a^2\rangle}=r^{\langle a^2\rangle}_{\{e\}}$,\\
$m^{\langle a^2\rangle}r_{\langle a^2\rangle}^{C_4}= r_{\langle a^2\rangle}^{C_4}m_{\im}^{C_4}$,\quad $ i_{\langle a^2\rangle}^{C_4}m^{\langle a^2\rangle}=m_{\im}^{C_4}i_{\langle a^2\rangle}^{C_4}$;\quad $ c_a^{\langle a^2\rangle}m^{\langle a^2\rangle}=m^{\langle a^2\rangle}c_a^{\langle a^2\rangle}$.
\end{enumerate}
There are two classes of double cosets in $[\{e\}\backslash\langle a^2\rangle /\{e\}]$ and $[\langle a^2\rangle\backslash C_4/\langle a^2\rangle]$ with set of representatives $ \{1,a^2\}$ and $ \{1,a\}$ respectively. Therefore, relation~\ref{rel:six} yields
\begin{enumerate}[resume]
\item \label{c4rel:three} $ r^{\langle a^2\rangle}_{\{e\}}i^{\langle a^2\rangle}_{\{e\}} = 1_{C(C_4)}+(c_a^{\{e\}})^2$,\quad $r_{\langle a^2\rangle}^{C_4}i_{\langle a^2\rangle}^{C_4}=1_{C(C_4/\langle a^2\rangle)}+c_a^{\langle a^2\rangle}.
   $
 \end{enumerate}
Finally, to use the last relation~\ref{rel:eleven}, one needs to identify the induced representations of the trivial representations $\mathrm{ind}_{\{e\}}^{\langle a^2\rangle}\tau^{\{e\}}$, $\mathrm{ind}_{\langle a^2\rangle}^{C_4}\tau^{\langle a^2\rangle}$ and of $\mathrm{ind}_{\langle a^2\rangle}^{C_4}\chi^{\langle a^2\rangle}.$ After completing this exercise in linear algebra, one finds that $\mathrm{ind}_{\{e\}}^{\langle a^2\rangle}\tau^{\{e\}} = \tau^{\langle a^2\rangle}+\chi^{\langle a^2\rangle}$, $\mathrm{ind}_{\langle a^2\rangle}^{C_4}\tau^{\langle a^2\rangle} =\tau^{C_4}+(\chi_{\im}^{C_4})^2$ and $\mathrm{ind}_{\langle a^2\rangle}^{C_4}\chi^{\langle a^2\rangle}=\chi_{\im}^{C_4}+(\chi_{\im}^{C_4})^3$. So, by relation~\ref{rel:eleven} and additive part of ~\ref{rel:seven}, we have
\begin{enumerate}[resume]
\item \label{c4rel:four} 
$i^{\langle a^2\rangle}_{\{e\}}r^{\langle a^2\rangle}_{\{e\}}=1_{C(C_4/\langle a^2\rangle)}+m^{\langle a^2\rangle}$,\quad $i_{\langle a^2\rangle}^{C_4}
 r_{\langle a^2\rangle}^{C_4}= 1_{\C}+(m_{\im}^{C_4})^2$,\\
 $ i_{\langle a^2\rangle}^{C_4}m^{\langle a^2\rangle}
 r_{\langle a^2\rangle}^{C_4}=m_{\im}^{C_4}+(m_{\im}^{C_4})^3.$
\end{enumerate}
The generators~$\Gamma$ and relations~\ref{c4rel:zero}, \ref{c4rel:one}, \ref{c4rel:two}, \ref{c4rel:three} and~\ref{c4rel:four} define the ring $R^{C_4}$. 

\begin{corollary}
The functor \[\mathrm{k}_*^G\colon \B^G\to \mathfrak{Mod}_{\textup c}^{\Z/2}(R^{C_4}),\qquad A\mapsto \{\mathrm{K}_{\epsilon}^H(\operatorname{Res}^G_H A)\}^{H\subseteq G}_{\epsilon\in \Z/2}\] into the abelian category of $\Z/2$\nbd-graded countable right modules over the ring $R^{C_4}$ is the universal stable $\ker\mathrm{k}_*^G$\nbd-exact functor.
\end{corollary}
\begin{proof}
By Remark~\ref{rem:res}, \[\mathrm{K}_{\epsilon}^H(\operatorname{Res}^G_H A)\cong \mathrm{KK}^G_{\epsilon}(C(G/H),A).\] By~\cite[Theorem 4.9]{Del14}  and computations in this subsection, $R^{C_4}$ is the category ring of $\mathfrak{C}^{C_4}$, thus we arrive at the result as  explained in Subsection~\ref{sub:uni}.
\end{proof}
The ring $R^{C_4}$ is not hereditary, thus there is no Universal Coefficient Theorem. However, as a special case of general theory, there is still a spectral sequence that relates relative derived functors on $\mathfrak{B}^G$ to the derived functors on $\mathfrak{Mod}_{\textup c}^{\Z/2}R^{C_4}$~\cite{Del14}. This can be used for classification purposes. Another direction is to localize the ring in question at different subsets to arrive at the hereditary situation, but this is out of the scope of the current exposition.
\subsubsection{Outlook}
The situation gets more complicated when when group $G$ has non\nbd-trivial projective representations. In this case, the equivariant bootstrap class has generators given by $G$\nbd-C*\nbd-algebras induced from subgroup actions on matrices of dimension bigger than one. Mackey\nbd-like relations still arise, however, in addition, one has to take into account the non\nbd-trivial $2$\nbd-cocycles in $\mathrm{H}^2(H,\mathrm{U}(1))$ for $H\subseteq G$.

\end{document}